\documentclass[12pt]{amsart}

\usepackage{amsmath} 
\usepackage{amsfonts}
\usepackage{amssymb}
\usepackage{amsthm}
\usepackage{mathtools}
\usepackage{enumerate}
\usepackage{setspace}
\usepackage{tikz}
\usepackage{graphicx}
\usepackage{tikz}
\usetikzlibrary{matrix}
\usepackage{accents}

\usepackage{url}

\usepackage{float}

\newtheorem{definition}{Definition}[section]
\newtheorem{lemma}{Lemma}[section]
\newtheorem{thm}{Theorem}[section]

\newtheorem{coro}{Corollary}[section]
\newtheorem{remark}{Remark}[section]
\newtheorem{thm*}{Theorem}[section]
\numberwithin{equation}{section}

\newcounter{example}[section]

\newcommand{\pr}{\partial}
\newcommand{\p}{\partial}

\newcommand{\lm}[2]{\lim\limits_{#1\to #2}}
\newcommand{\ovr}[1]{\overline{#1}}   

\newcommand{\be}{\begin{equation}}
\newcommand{\ee}{\end{equation}}
\newcommand{\bee}{\begin{equation*}}
\newcommand{\eee}{\end{equation*}}

\def\tr{\textnormal{tr}}
\def\dv{\textnormal{div}}

\def\dint{\displaystyle\int}
\def\R{\mathbb{R}}
\def\vol{\mathrm{vol}\,}

\def\S{\Sigma}
\def\({\left(}
\def\){\right)}

\def\To{\longrightarrow}

\def\s{\sigma}

\def\w{\omega}

\def\V{\mathcal{V}}

\def\graph{\textnormal{graph}}

\def\bM{\mathbf{M}}
\def\ADM{\textnormal{ADM}}
\def\m{\mathbf{m}}

\def\defeq{\coloneqq}
\def\A{\mathcal{A}}

\def\BY{\textnormal{BY}}
\def\FH{\accentset{\circ}{H}}

\allowdisplaybreaks

\title[Stability of a quasi-local PMT for graphs]{Stability of a quasi-local positive mass theorem for graphical hypersurfaces of Euclidean space}
\author[Alaee]{Aghil Alaee}
\address{Center of Mathematical Sciences and Applications\\
	Harvard University\\
	Cambridge, MA 02138, USA}
\email{aghil.alaee@cmsa.fas.harvard.edu}

\author[{Cabrera Pacheco}]{Armando J. {Cabrera Pacheco}}
\address{Department of Mathematics, Universit\"at T\"ubingen,  72076 T\"{u}bingen, Germany.}
\email{cabrera@math.uni-tuebingen.de}

\author[McCormick]{Stephen McCormick}
\address{Matematiska institutionen, Uppsala universitet, 751 06 Uppsala, Sweden.}
\email{stephen.mccormick@math.uu.se}

\subjclass[2010]{53C20, 83C99}

\begin{document}

\begin{abstract}
We present a quasi-local version of the stability of the positive mass theorem. We work with the Brown--York quasi-local mass as it possesses positivity and rigidity properties, and therefore the stability of this rigidity statement can be studied. Specifically, we ask if the Brown--York mass of the boundary of some compact manifold is close to zero, must the manifold be close to a Euclidean domain in some sense?

Here we consider a class of compact $n$-manifolds with boundary that can be realized as graphs in $\R^{n+1}$, and establish the following. If the Brown--York mass of the boundary of such a compact manifold is small, then the manifold is close to a Euclidean hyperplane with respect to the Federer--Fleming flat distance.
\end{abstract}

\maketitle
\thispagestyle{empty}

\vspace{-5mm}

\section{Introduction}
The positive mass theorem is one of the most celebrated results in mathematical general relativity, proven independently by Schoen and Yau \cite{Schoen-Yau79,SY-17}, and by Witten \cite{Witten81}. In the context of time-symmetric initial data, it states that a complete asymptotically flat $n$-manifold with non-negative scalar curvature has non-negative ADM mass, where the ADM mass is a geometric quantity associated to the asymptotics of the manifold. From the perspective of physics, such an asymptotically flat manifold represents an isolated gravitating system and the ADM mass is the total mass (energy) contained in the system. The positive mass theorem also has an associated rigidity property; namely, the mass is in fact strictly positive unless the manifold is isometric to $\R^n$. A natural question that then arises is the associated stability question: \emph{If the ADM mass is small, is the manifold close to Euclidean space in some sense?}

That is, given a sequence of asymptotically flat manifolds with non-negative scalar curvature whose ADM masses converge to zero, does this sequence converge in some topology to Euclidean space? For many natural topologies, the answer is no. One can see this by considering so-called ``gravity wells" -- geometrically, these are long thin fingers which can be taken to be getting longer and more thin along the sequence of manifolds (See Figure \ref{fig-gravwells} below). Such a sequence of manifolds may not converge to Euclidean space under many usual notions of convergence of manifolds (see Example 2.9 of \cite{LeeSormani}), and indeed one can construct such examples explicitly where the ADM mass can be made to converge to zero.

Despite this, there has been some success in proving stability under additional hypotheses or using notions of convergence that don't ``see" these gravity wells in the limit. A good notion of convergence for dealing with these gravity wells is intrinsic flat convergence \cite{SW11}, and within a class of spherically symmetric manifolds Lee and Sormani \cite{LeeSormani} were able to prove that the positive mass theorem is stable with respect to the intrinsic flat topology. Huang and Lee \cite{HuangLee2015} were then able to prove stability of the positive mass theorem for manifolds that can be realized as graphs in Euclidean space one dimension higher, with respect to the flat distance in the ambient Euclidean space. The corresponding problem with respect to the intrinsic flat distance was studied by Huang, Lee and Sormani \cite{HLS}. For a class of axially symmetric manifolds, Bryden \cite{Bryden} proved stability of the positive mass theorem with respect to a Sobolev topology, after excising a neighborhood of the central axis. The symmetry combined with the excision of the central axis is what excludes the possibility of gravity wells forming in the domain considered, and is what allows for convergence in a stronger topology. Similarly Lee proved stability, with respect to uniform convergence, of an exterior region for harmonically flat manifolds \cite{Lee}. Another class of manifolds where stability holds with respect to a Sobolev topology are those considered by Allen \cite{Allen1,Allen2}. He considers manifolds that are foliated by a smooth inverse mean curvature flow, with two boundary components that are leaves of the flow. Stability of the positive mass theorem has also been studied for asymptotically hyperbolic manifolds, by Allen \cite{Allen3}; Dahl, Gicquad and Sakovich \cite{D-G-S-stability}; Sakovich and Sormani \cite{S-S}; and by the second-named author \cite{CabreraPacheco19}. Further related stability results can be found in \cite{AB19,B-F,BKS19,Corvino05,Finster,FK}.

In this article we study the same kind of stability problem, however we are interested in the stability of a positive mass theorem for quasi-local mass. There is a wealth of literature on different definitions of quasi-local mass, and the interested reader is directed to \cite{QLMLRR} for a survey of the different definitions in the literature. One definition of quasi-local mass that does have a good positive mass theorem with rigidity is the Brown--York mass, whose positivity was proven by Shi and Tam \cite{ShiTam02}. Note that the problem of ``gravity wells" mentioned above exists here too. One can construct a sequence of compact manifolds with boundary, with Brown--York mass of the boundary going to zero, that do not converge to a smooth manifold under many common topologies.

For a class of compact manifolds with boundary that can be realized as graphs over $\R^n$ in $\R^{n+1}$, we prove that the Brown--York positive mass theorem is stable with respect to the Federer--Fleming flat distance \cite{FF60}. The main theorem refers to a class of compact manifolds with boundary that can be realized as graphs in $\R^{n+1}$ over a connected set $\ovr{U} \subset \R^n$, which we denote by $\mathcal{A}_U^{D}$. For the precise definition of $\mathcal{A}_U^{D}$, the reader is referred to Definition \ref{def-family}. Other technical definitions are also reserved for Section \ref{Sdefs} in order to first state the main result of this article in a simplified way (Theorem \ref{thm-main}). The reader is referred to Theorem \ref{thm-main2} for the precise statement of the result.

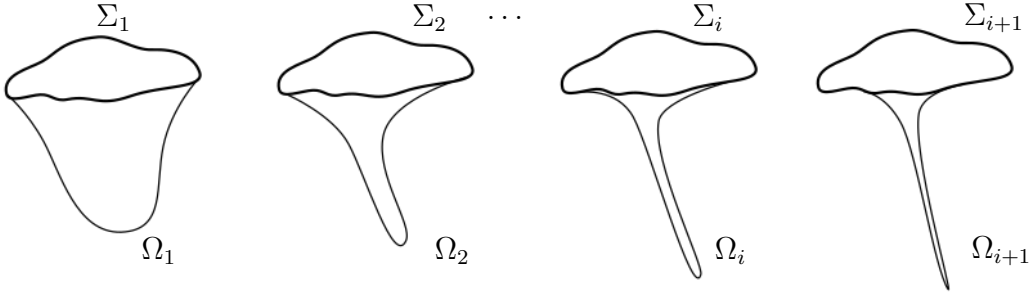
\begin{figure}[H]
	
	\begin{tikzpicture}
	\node[anchor=south west,inner sep=0] at (0,0) {\includegraphics[scale=.8]{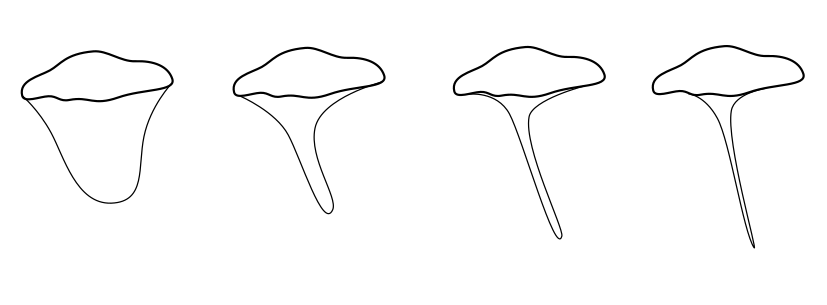}};
	\node at (1.8,4.6) {$\S_1$};
	\node at (2.4,1.5) {$\Omega_1$};
	\node at (6,4.6) {$\S_2$};
	\node at (6.3,1.5) {$\Omega_2$};
	\node at (7,4.6) {$\hdots$};
	\node at (9.7,4.6) {$\S_i$};
	\node at (10,1.5) {$\Omega_i$};
	\node at (13.5,4.6) {$\S_{i+1}$};
	\node at (13.6,1.5) {$\Omega_{i+1}$};
	\end{tikzpicture}
	\caption{A sequence of manifolds with increasingly longer and thinner ``gravity wells".}
	\label{fig-gravwells}
\end{figure}

\begin{thm} \label{thm-main}
	Let $(\Omega_i,g_i)$ be a sequence of a manifolds in $\mathcal{A}_U^{D}$ and denote by $\S_i$ the boundary component of $\Omega_i$ that is not a minimal surface. After an appropriate normalization, if the limit of the Brown--York mass $\m_{\BY}(\S_i)$ goes to zero, then $(\Omega_i,g_i)$ converges to $\{0\} \times \ovr{U}$ with respect to the Federer--Fleming flat distance. Moreover, $\vol(\Omega_i) \to \vol(U)$ as $\m_{\BY}(\S_i) \to 0$.
\end{thm}
We remark that the definition of $\mathcal{A}_U^{D}$ permits the boundary of each $\Omega_i$ to have at most one connected component that is not a minimal surface. The normalization mentioned above in Theorem \ref{thm-main} consists of a vertical translation of the isometric embedding of $(\Omega_i,g_i)$ as a graph in $\R^{n+1}$.

To the best of our knowledge the problem of the stability of domains in Euclidean space as unique minimizers of quasi-local mass has yet to be directly addressed. The only result in this direction that we are aware of is a result of Allen \cite{Allen2} which we paraphrase as follows. Let $(\Omega,g_i)$ be a sequence of compact manifolds with connected boundary $\S_i$ and non-negative scalar curvature that are foliated by smooth inverse mean curvature flow with $\S_i$ a leaf of this foliation. Assume further that the Hawking mass of each $\S_i$ is non-negative and converging to $0$, then $g_i$ converges to the flat metric with respect to an appropriate Sobolev norm. The reader is directed to Theorem 3.1 of \cite{Allen2} for the precise statement of the result. However, since the Hawking mass can be quite negative in general, one cannot expect a result like this to hold outside of special cases.

We pursue this problem following the approach of Huang and Lee \cite{HuangLee2015} in the case of the stability of the positive mass theorem for asymptotically flat manifolds. We note that this approach has been applied in different settings: to establish the stability of the positive mass theorem for asymptotically hyperbolic graphs by the second-named author \cite{CabreraPacheco19}  and  also of graphical tori with almost non-negative scalar curvature by the Ketterer, Perales, and the second-named author \cite{CKP}. While we have focused on this particular approach, we expect that one could prove a similar quasi-local stability result for the Brown--York mass in any other case where stability of the positive mass theorem holds for asymptotically flat manifolds. We also would like to remark that in \cite{CKP}, the dimension was restricted to $n=3$, however each of the arguments establishing Theorem 1.3 therein are valid for arbitrary dimension, with the exception of Lemma 4.6 there. We remark that Lemma \ref{lemma-supf-est} below includes a higher dimensional estimate analogous to Lemma 4.6 of \cite{CKP}, which would promote Theorem 1.3 of \cite{CKP} to all dimensions $n\geq3$.

The structure of this article is as follows. In Section \ref{Sdefs} we discuss some properties of the Brown--York quasi-local mass, and give precise definitions of the graphical manifolds considered here. We also show that a positive mass stability type result due to Corvino \cite{Corvino05} has a straightforward analogue for the stability of the Brown--York mass positivity. In Section \ref{section-vol-est} we obtain volume estimates that are required for the proof of the main theorem. Then in Section \ref{Sflatconvergence} we provide the proof of Theorem \ref{thm-main}.

\medskip

\noindent {\bf Acknowledgements.} AA acknowledges the support of the Gordon and Betty Moore Foundation and the John Templeton Foundation. AJCP is grateful to the Carl Zeiss Foundation for its generous support and to the Center of Mathematical Sciences and Applications at Harvard University for its hospitality during his research visit in August 2019. SM would like to thank Institut Mittag-Leffler (Djursholm, Sweden) for hospitality where part of this research was conducted in the autumn of 2019. The authors would also like to thank Lan Hsuan-Huang and Christina Sormani for helpful discussions during the preparation of this article.

\section{The Brown--York mass and graphical manifolds} \label{Sdefs}
 The Brown--York mass is a quasi-local notion of energy in general relativity that is usually defined for surfaces in $3$-manifolds. Given a closed surface $\S$ with positive Gauss curvature in a Riemannian manifold $M$, the Brown--York mass is defined as
\be 
\m_{\BY}(\S):=\frac{1}{8\pi}\int_\S \left(H_o-H\right)\,dS,
\ee 
where $H$ is the mean curvature of $\S$ in $M$ and $H_o$ is the mean curvature of the unique isometric embedding of $\S$ in $\R^3$ guaranteed by classical results from Nirenberg \cite{Nirenberg} and Pogorelov \cite{Pogorelov}. Here and throughout, we use the convention that the mean curvature with respect to the outward unit normal $\nu$ is taken to be $H:=\tr(\Pi)=\dv(\nu)$, equal to $2$ for the unit sphere in $\R^3$, where $\Pi_{ij}:=\nabla_i\nu_j$ denotes the second fundamental form. With respect to the same unit normal, we define the mean curvature vector to be $\vec{H}=-H\nu$.

If $\S$ is strictly mean convex and the ambient manifold has non-negative scalar curvature, then it was shown by Shi and Tam \cite{ShiTam02} that the Brown--York mass is non-negative and only zero if $\S$ bounds a domain isometric to a domain in Euclidean space. In fact they showed that, regardless of dimension, assuming an appropriate isometric embedding exists then the Brown--York mass is positive under a spin assumption. Furthermore, as the positive mass theorem has now been established by Schoen and Yau in all dimensions \cite{SY-17} without the spin assumption, Miao's version of the positive mass theorem with corners \cite{Miao} extends to all dimensions and therefore the result of Shi and Tam can be stated as follows.

\begin{thm}[\cite{ShiTam02}, Theorem 4.1] \label{thm-ShiTam}
	Let $(\Omega,g)$ be a compact $n$-manifold with non-negative scalar curvature and connected boundary with positive (outward-pointing) mean curvature $H$. Assume further that $\p \Omega$ is isometric to strictly convex surface in $\R^{n}$ with mean curvature $H_o$, then
	\be \label{eq-BYpos}
	\int_{\p\Omega} H_o \, dS \geq \int_{\p\Omega } H\, dS.
	\ee 
	Furthermore, in the case of equality in \eqref{eq-BYpos} then $(\Omega,g)$ is isometric to a domain in $\R^n$.
\end{thm}
We remark that the hypothesis of strict convexity in Theorem \ref{thm-ShiTam} was relaxed by Eichmair, Miao and Wang \cite{EMW} to the condition that $\p\Omega$ have positive scalar curvature and that each boundary component be isometric to a mean convex star-shaped surface in $\R^n$.

In addition, the Brown--York mass also satisfies a quasi-local Penrose inequality in all dimensions that the standard Riemannian Penrose inequality holds. This was shown for $n=3$ by Shi and Tam \cite{ShiTam07}, and follows straightforwardly in higher dimensions by combining the Riemannian Penrose inequality with corners \cite{Miao09} (see also \cite{MM19}) with Shi and Tam's proof of the positivity of the Brown--York mass \cite{ShiTam02}. Furthermore the inequality is in fact strict provided that the horizon is non-empty, as noted by Lu and Miao \cite{LM}.
\begin{thm}\label{thm-QLpenrose}
	Let $(\Omega,g)$ be a compact Riemannian $n$-manifold ($3\leq n \leq 7$) with smooth boundary $\p\Omega$ and non-negative scalar curvature. Suppose $\p\Omega=\S_H\cup\S$ where $\S_H$ is a closed (possibly disconnected) minimal surface, and $\S$ has positive (outward) mean curvature and is isometric to a strictly convex closed embedded sphere in $\R^n$. Denote by $H$ the mean curvature of $\S$ in $\Omega$ and by $H_o$ the mean curvature of its isometric embedding into $\R^n$, then we have
	\be \label{eq-BYpenrose}
	\frac{1}{(n-1)\omega_{n-1}}\int_{\S}\left(H_o-H\right)\,dS>\frac12\( \frac{|\S_H|}{\omega_{n-1}} \)^{\frac{n-2}{n-1}}.
	\ee 
\end{thm}

Note that for $n=3$ the condition that $\S$ have positive Gauss curvature is sufficient to ensure the existence of a unique isometric embedding into $\R^3$, which indeed is strictly convex.

We now turn to our first stability-type result, which follows easily from the observation that Theorem 2.3 of \cite{Corvino05} can be localized to any quasi-local mass for which an appropriate quasi-local Penrose inequality holds. We show that if the Brown--York mass of the boundary of a compact manifold is small relative to the sectional curvatures of the manifold, then the manifold must be diffeomorphic to a ball. For a special case of the spacetime version of this theorem the reader is referred to \cite{Aghil}.

\begin{thm}
	Let $(\Omega,g)$ be a compact Riemannian manifold with smooth boundary and non-negative scalar curvature. Suppose that the boundary has positive Gaussian curvature, positive mean curvature, and that the sectional curvatures are bounded above by some constant $C^2$, $C>0$. Then if the Brown--York mass of the boundary satisfies $\m_{\BY}(\pr \Omega)< \frac{1}{2C}$ then $\Omega$ contains no closed minimal surfaces and is diffeomorphic to a ball in $\mathbb{R}^3$.
\end{thm}
\begin{proof}
	We follow Corvino's argument almost identically \cite{Corvino05}, replacing the ADM mass with the Brown--York mass. Note that the boundary must be a $2$-sphere since it has positive Gaussian curvature.  For the sake of obtaining a contradiction, consider $\m_{\BY}(\pr \Omega) $ and $C$ as in the statement of the theorem and assume $\Omega$ is not diffeomorphic to a ball. By a classical result of of Meeks, Simon and Yau (see Theorem 1$''$ and Proposition 1 of \cite{MeeksSimonYau}), there exists a closed minimal surface $\S\cong\mathbb{S}^2$ in $\Omega$. For a given point $p\in\S$, let $\{e_1,e_2\}$ be a basis of $T_p\S$ in which the second fundamental form $\Pi$ is diagonal and denote the principal curvatures at $p$ by $\kappa_1=\Pi(e_1,e_1)$ and $\kappa_2=\Pi(e_2,e_2)$. The Gauss equation then gives the Gauss curvature $K$ in terms of the Riemann tensors $R$ and $\overline{R}$ of $\S$ and $\Omega$ respectively,
	\be 
	K=R(e_1,e_2,e_1,e_2)=\overline{R}(e_1,e_2,e_1,e_2)+\kappa_1\kappa_2.
	\ee 
	Since $\S$ is minimal, $\kappa_1\kappa_2\leq0$ at $p$ and therefore we have
	\bee 
	K\leq C^2,
	\eee
	and indeed this holds everywhere on $\S$. By the Gauss--Bonnet Theorem we have
	\be 
	4\pi=\int_\S K \leq |\S|C^2<\frac{|\S|}{4\m_{\BY}(\p\Omega)^2},
	\ee 
	and by Theorem \ref{thm-QLpenrose} we then have
	\be 
	\m_{\BY}(\p\Omega)< \(\frac{|\p \Omega|}{16\pi}\)^{1/2}\leq \m_{\BY}(\p\Omega);
	\ee 
	that is, no such closed minimal surface can exist and it is a contradiction. Therefore, $\Omega$ is diffeomorphic to a ball in $\mathbb{R}^3$.
\end{proof}

We now turn towards the case of graphical manifolds, and in particular, give some definitions that will be needed throughout.

\begin{definition}
	Let $U_o \Subset U$ be two bounded open sets in $\R^n$, with $U$ connected and $U_o$ possibly empty. A compact graphical manifold with boundary is the graph of a smooth function $f:\ovr{U} \setminus U_o \To \R$ in $\R \times \ovr{U}\subset \R^{n+1}$. If $U_o \neq \emptyset$, then we assume that $f(\pr U_o)$ is a minimal boundary, that is, $f$ locally constant on $\pr U_o$ and $|Df| \to \infty$ on $\pr U_o$. In the case that $U_o \neq \emptyset$, we say that $\graph[f]$ (or simply $f$) has a minimal boundary, otherwise we say that it is entire.
\end{definition}

In this work we are interested in compact Riemannian $n$-manifolds that can be isometrically embedded into $\R^{n+1}$ as graphical manifolds with boundary. More precisely we define the following family of manifolds.

\begin{definition} \label{def-family}
	We say that $(\Omega,g)$ is an admissible compact graph over $U$, and denote it $(\Omega,g) \in \mathcal{A}_U$, if it is a Riemannian manifold with boundary that can be isometrically embedded in $\R^{n+1}$ as the compact graphical manifold with boundary of some function $f:\overline{U}\setminus U_o\to \R$. In addition, we require that $f$ satisfies:
	\begin{enumerate}[(i)]
		\item the scalar curvature of $\graph[f]$, denoted $R(f)$, is non-negative,
		\item almost every level set $f^{-1}(h)$ is outer-minimizing and strictly mean convex, and
		\item $f$ is locally constant on $\pr (\ovr{U} \setminus U_o)$.
	\end{enumerate}
	We say that  $(\Omega,g) \in \mathcal{A}_U^{D}$ if in addition $|\max(f)| \leq D$.
	
\end{definition}

\begin{remark}
The hypotheses in Definition \ref{def-family} arise naturally from the stability of the positive mass theorem for graphical manifolds by Huang and Lee \cite{HuangLee2015}. Here, we also consider the possibility of $|\max(f)|$ to be bounded since in this case the whole family $\A_U^D$ can be embedded into the fixed compact set $[-D,D] \times  \ovr{U}$.
\end{remark}

\begin{remark}
By the assumptions in Definition \ref{def-family}, given such an $f:\overline{U}\setminus U_o\to \R$, the Brown--York mass of $\pr U$, $\m_{\BY}(\pr U)$, is well-defined in terms of the graph function $f$. Throughout, we will simply write $\m_{BY}$ when referring to hypersurfaces in compact graphs, however it should be understood that $H_o$ is taken to come from the isometric embedding given by the graph. As an aside, we mention that there exist generalisations of the Cohn-Vossen Rigidity Theorem \cite{Cohn-Vossen} in higher dimensions (see, for example \cite{DR-rigidity}, so one could impose conditions to make sure any isometric embedding is unique up to rigid motions, however here we focus on isometric embeddings that are part of graphical manifolds.
\end{remark}

Our first result on graphs is to demonstrate that without loss of generality, we can assume that the mean curvature vector of the graphical manifold in $\R^{n+1}$ points upwards. Throughout, we compute the mean curvature of the graphical manifolds $\widehat H$ in $\R^{n+1}$ with respect to the downward-pointing unit normal. The following theorem is a localized version of \cite[Theorem 3]{HuangWu2015}.

\begin{thm}\label{thm-MCgraph}
	Assume $U \subset \R^n$ is a bounded open connected set and  $U_o \Subset U$ is open (not necessarily connected). Let $M$ denote the graph of $f\in C^{n+1}(\ovr{U}\backslash U_o)$ with boundaries $\Sigma_o$ and $\Sigma$ corresponding to $\partial U_o$ and $\partial U$, respectively. Suppose each connected component of $\partial U\cup\partial U_o$ is a level set of $f$ with $|Df|\to\infty$ as $x\to\partial U_o$. In addition, suppose that $\pr U$ is strictly mean convex in $\R^n$. If the scalar curvature of the graph of $f$ is non-negative, then its mean curvature $\widehat H$ has a sign.
\end{thm}

\begin{proof}
By Corollary 2.3 of \cite{HuangWu2013}, we know that $\widehat H > 0$ on $f(\pr U)$, so we assume that the mean curvature is negative somewhere in order to arrive at a contradiction. Let $M_+$ be a connected component of $\{p\in M: \widehat H\geq 0\text{ at $p$}\}$ that contains a neighbourhood of $f(\pr U)$. Note that the graph $M$ is homeomorphic to $U\backslash U_o$ given by the projection $\pi(x, f(x)) = x$, and write $\ovr{U}\backslash \pi(M_+)$ as the disjoint union $\cup_{\alpha}V_{\alpha}$ where each $V_{\alpha}$ is connected. Clearly such a $V_\alpha$ must exist by our assumption that the mean curvature is negative somewhere. Because $\pi(M_+)$ is connected, each $\partial V_{\alpha}$ is connected by Proposition A.3 of \cite{HuangWu2013}, noting that $V_\alpha$ is separated from $f(\pr U)$ by construction. Now fix some $V_\alpha$ and set $\Gamma=\pi^{-1}(\partial V_{\alpha})$. By the Gauss equation $0\leq R=\widehat H^2-|\widehat \Pi|^2$, where $\widehat \Pi$ is the second fundamental form of graph $M$ in Euclidean space, we infer that $\Gamma$ is a subset of interior geodesic points of $M$; that is, $\Gamma\subset\{p\in M, \widehat \Pi=0\,\text{at $p$}\}$. Moreover, it follows from \cite[Lemma 3.6]{HuangWu2013} that $\Gamma$ lies in a hyperplane $\Xi$. Note that $\Gamma$ does not intersect $\Sigma_o$ because $|Df|\to\infty$ on $\Sigma_o$. Hence $M$ can be represented as the graph of a $C^{n+1}$-function $u$ in an open neighborhood of $\Gamma$ in $\Xi$, and $u$ satisfies $u = 0$, $|Du| = 0$, $|D^2u| = 0$ on $\Gamma$. But since $V_{\alpha}$ is bounded and the scalar curvature $M$ is non-negative, by the maximum principle for the mean curvature $\widehat H$ we get $u\equiv 0$ on $\pi^{-1}(V_{\alpha})$, see \cite[Theorem 3.9]{HuangWu2013}. Hence the mean curvature vanishes on $\pi^{-1}(V_{\alpha})\subset M$. It follows that $\pi^{-1}(V_{\alpha})\subset M_+$ which is a contradiction. 
\end{proof}

From now on we denote the level sets of a compact graph $f:\ovr{U} \setminus U_o \To \R$ by
\bee
\S_h \defeq \{ f = h \},
\eee
for a regular value $h$ and set $\S \defeq \S_{\max(f)}$ and $\S_o \defeq f(\pr U_o)$, where $\pr U_o$ is the minimal boundary.

\begin{figure}[H]
	
	\begin{tikzpicture}
	\node[anchor=south west,inner sep=0] at (0,0) {\includegraphics[scale=.9]{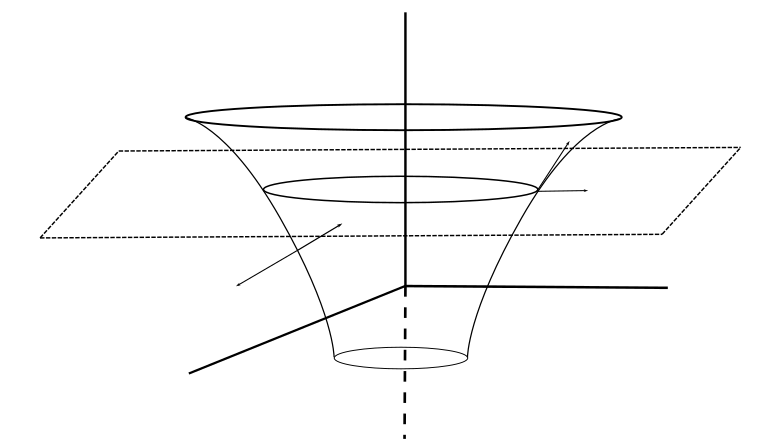}};
	\node at (7.2,4.3) {$-\widehat{H}\widehat{N}$};
	\node at (4.2,2.8) {$\widehat{N}$};
	\node at (11.6,4.8) {$n$};
	\node at (4.6,4.7) {$\S_h$};
	\node at (10.8,5.9) {$\nu$};
	\node at (13.7,3.7) {$\{x^{n+1}=h\}$};
	\node at (12.6,2.5) {$\R^n$};
	\node at (7.8,8.8) {$x^{n+1}$};
	\node at (12.8,6.5) {$\graph[f]$};
	
	\end{tikzpicture}
	\caption{A graphical manifold showing conventions used throughout.}
	\label{fig-mean-c-conv}
\end{figure}
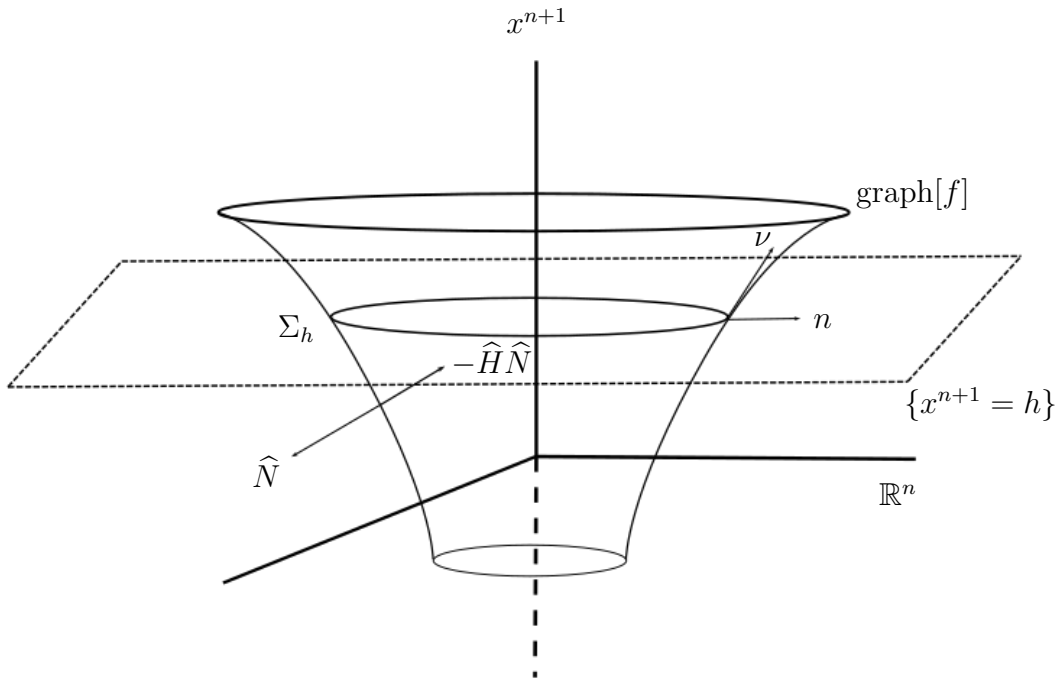

Throughout, we use the notation $H_{\S_{h}}$ to denote the outward mean curvature of the level set $\S_h$ in $\graph[f]$ and $\FH_{\S_{h}}$ to denote the outward mean curvature of the same surface as a hypersurface in $\R^n\cong\{ x_{n+1}=h \}\subset \R^{n+1}$.

Theorem \ref{thm-MCgraph} can be interpreted as follows (cf. \cite{HuangLee2015}). If a compact graphical manifold has non-negative scalar curvature,  then without loss of generality we can assume that its mean curvature vector points upward (see Figure \ref{fig-mean-c-conv} above). This in turn implies that the level sets $\S_h$ (for regular values) are weakly mean convex with respect to the outward normal unit vector (in the hyperplane $\R^n$ containing them) (cf. \cite[Corollary 2.11]{HuangLee2015}).

Given a compact graph $f:\ovr{U} \setminus U_o \To \R$, denote by $\ovr{f}$ the function obtained by extending $f$ to $\ovr{U}$ in such a way that $\ovr{f}$ is constant on $\ovr{U_o}$. Note that this is possible since by assumption $f$ is locally constant on $\pr U_o$. Define the set 
\bee
\Omega_h \defeq \{ x \in \ovr{U} \, | \, f(x) < h \},
\eee
and denote by $\S_h$ the reduced boundary $\pr^*\Omega_h$. Note that almost every $h \in [\min(f),\max(f)]$ is a regular value by Sard's theorem, and when $h$ is a regular value, $\S_h$ is precisely the level set $\{ f = h \}$ as defined above. We define the function
\bee
\V(h) = |\S_h|,
\eee
where $|\cdot|$ denotes the $(n-1)$-dimensional Hausdorff measure. From \cite{HuangLee2015} (cf. \cite{CKP}), we have the following lemma.

\begin{lemma}\label{lemma-Vincr}\cite[Lemma 3.3]{HuangLee2015}
Let $f:\ovr{U} \setminus U_o \To \R$ be a compact graphical manifold such that $\S_h$ is outer-minimizing for almost every $h \in f(\ovr{U})$. Then, $\V$ is lower semicontinuous, non-decreasing, and $\V(h) \leq |\pr U|$ for all $h \in f(\ovr{U})$.
\end{lemma}

\begin{remark}
We warn the reader about a slight abuse of notation. Notice that $\S$ has two meanings: a closed hypersurface of manifold $M$ and the level set $\S = \S_{\max(f)}$ of a compact graphical manifold $\graph[f]$. Although this abuse of notation is undesirable, no confusion should arise since its meaning will always be clear from context, and otherwise it will be explicitly stated. Also, for a closed hypersurface $\S$ in a manifold $M$, it is common to denote by $H$ and $H_o$ its mean curvatures in $M$ and $\R^n$ (whenever $\S$ can be isometrically embedded into $\R^n$), respectively. In the case of a level set $\S_h$ in a graph $\graph[f]$, we will always use $H_{\S_h}$ and $\FH_{\S_h}$ to denote its mean curvatures in $\graph[f]$ and $\R^n$, respectively.
\end{remark}

Next, let us recall some observations of Lam \cite{Lam} that follow from some standard computations. Since \cite{Lam} is only considering dimension 3, we repeat the proof here in higher dimensions for completeness.

\begin{lemma}[\cite{Lam}]  \label{lemma-mBYbound}
Let $\Omega$ be a compact manifold with boundary $\pr \Omega$ that has positive outward mean curvature and is isometric to a strictly convex closed embedded sphere in $\R^n$. Let $H$ be the mean curvature of $\pr \Omega$ in $\Omega$ and $H_o$ be its mean curvature when isometrically embedded into $\R^n$. Then, 
\be \label{eq-BYLamestimate}
\m_{\BY}(\pr \Omega)\geq \frac{1}{2(n-1)\omega_{n-1}}\int_{\pr \Omega} H_o\left( 1-\frac{H^2}{H_o^2} \right) \, d\s.
\ee 
If we additionally impose non-negative scalar curvature on $\Omega$, equality holds if and only if $\Omega$ is isometric to a domain in $\R^n$. 
Moreover, suppose that $f:\ovr{U} \setminus U_o \To \R$ is a compact graphical manifold and the level set $\S_h \defeq \{ f = h \}$, where $h$ is a regular value, is strictly mean convex, we have
\be \label{eq-mqlequality}
\int_{\S_h} \frac{|Df|^2}{1+|Df|^2} \FH_{\S_h} \, dV_h= \int_{\S_h} \FH_{\S_h} \(1 - \frac{H_{\S_h}^2}{\FH_{\S_h}^2}    \)\, dV_h.
\ee
In particular if $f$ is constant on $\pr U$,
\be \label{eq-BYLam2}
\m_{\BY} \geq \frac{1}{2(n-1)\omega_{n-1}}\int_{f(\pr U)}  \frac{|Df|^2}{1+|Df|^2} \FH_{\S}  \, dA_h,
\ee
where $\m_{BY}$ is the Brown--York mass of $f(\pr U)$ in the graphical manifold.
\end{lemma}

\begin{proof}
An elementary computation gives
\begin{align}
\begin{split}\label{eq-mbylam}
0\leq\int_{\S}\frac{1}{H_o}\(H_o-H \)^2\,d\s&=\int_\S H_o-2H+\frac{H^2}{H_o}\,dS\\
&=\int_\S 2\(H_o-H \)-H_o+\frac{H^2}{H_o}\,dS,
\end{split}
\end{align}
which implies \eqref{eq-BYLamestimate}. Note that equality in \eqref{eq-BYLamestimate} implies $H\equiv H_o$ and therefore rigidity follows by rigidity in the positive mass theorem with corners \cite{ShiTam02}.\\

Note that $n=n^i\partial_i=\frac{\delta^{ij}f_j}{|Df|}\partial_i$ is the unit outward normal of $\S_h$ in $(\mathbb{R}^n,\delta)$ and $\nu=\nu^i\partial_i=\frac{g^{ij}f_j}{|\nabla f|_g}\partial_i$ is the unit normal of $\Sigma_h$ in $(M,g=\delta+df^2)$. Then the mean curvature is given by \be \FH_{\S_h}=\text{div}_{\S_h}n=\frac{\Delta_{\mathbb{R}^n}f}{|Df|}-\frac{\text{Hess}f(Df,Df)}{|Df|^3}.
\ee 
A straightforward computation then leads to $H_{\S_h}=\text{div}_{\S_h}\nu=\frac{\FH_{\S_h}}{\sqrt{1+|Df|^2}}$, which gives 
\begin{align*}
\dint_{\S_h} \dfrac{|D f|^2}{1+|D f|^2} \FH_{\S_h} dV_h&=\dint_{\S_h} \( 1 - \dfrac{1}{1+|Df|^2} \) \FH_{S_h}\, dV_h \\
&=\dint_{\S_h} \(  1 - \dfrac{ \FH_{\S_h}^2}{ \FH_{\S_h}^2(1+|Df|^2)} \) \FH_{\S_h}\, dV_h  \\
&=\dint_{\S_h} H_{\S_h}\( 1 - \dfrac{H_{\S_h}^2}{\FH_{\S_h}^2} \) dV_h.
\end{align*}

Finally, combining \eqref{eq-mqlequality} and \eqref{eq-BYLamestimate} gives \eqref{eq-BYLam2}.

\end{proof}

\begin{remark}
	While the \eqref{eq-BYLamestimate} follows from an elementary computation, it is interesting to note that it is also exactly the estimate one obtains by applying Lemma 2.11 of \cite{ShiTam02}, setting the function $u$ therein equal to $\frac{H_o}{H}$ and noting that
	\bee
	M(r) := \int_{\S_r} H_o(1-u^{-2})\, d\sigma_r
	\eee 
	is monotonically increasing and converges to $2(n-1)\omega_{n-1}\pi\m_{ADM}$ as $r \to \infty$. That is, $M(r)=\int_{\S_r}H_o\( 1-\frac{H^2}{H_o^2}\)d\sigma$ and $\m_{BY}(\S)\geq \m_{ADM}$ in the Shi--Tam extensions. In particular, we observe that the quantity $\frac{1}{16\pi}\int_{\S_r}H_o\( 1-\frac{H^2}{H_o^2}\)d\sigma$ proposed by Lam \cite{Lam} as a quasi-local mass arises naturally here, without reference to graphical manifolds.
\end{remark}

From the above Lemma, some corollaries follow immediately. We briefly note them here.
\begin{coro}
	Let $\S$ be a closed hypersurface with positive outward mean curvature, in a Riemannian $n$-manifold such that the Brown--York mass of $\S$ with respect to any isometric embedding into $\R^{n}$ is negative, if any such isometric embeddings exist. Then there does not exist any neighborhood of $\S$ that can be isometrically embedded in $\R^{n+1}$ as a graph.
\end{coro}
In the following, when we speak of an asymptotically flat manifold, we mean it in the usual sense. That is, the asymptotics are sufficient to ensure that the ADM mass is well-defined, and in particular, the scalar curvature is integrable.
\begin{coro}
	Let $(M,g)$ be an asymptotically flat $n$-manifold and assume that there is a neighborhood of infinity $(M,g)$ that can be realized as a graph in $\R^{n+1}$ such that almost every level set $\S_h$ is strictly mean convex. Then $(M,g)$ has non-negative mass.
\end{coro}
\begin{proof}
	Let $\S_{h}$ be the level sets in the statement of the theorem with associated graph function $f$, and define $C_n=2(n-1)\w_{n-1}$.

By direct computation we have
		\begin{align*}
		\m_{\BY}(\S_h) &= \frac{2}{C_n} \int_{\S_h} (H_o - H) \, dV_{h} \\
		&=\frac{2}{C_n} \int_{\S_h} \(1 - \frac{1}{\sqrt{1+|Df|^2}}  \)\FH_{\S_h} \, dV_{h} \\
		&=\frac{2}{C_n} \int_{\S_h} \( \frac{\sqrt{1+|Df|^2}-1}{\sqrt{1+|Df|^2}}  \)\FH_{\S_h} \, dV_{h} \\
		&=\frac{2}{C_n} \int_{\S_h} \( \frac{|Df|^2}{1+|Df|^2}  \) \frac{\sqrt{1+|Df|^2}}{\sqrt{1+|Df|^2}+1}\FH_{\S_h} \, dV_{h} 
		\end{align*}

Then simply by taking the limit as these surfaces go out to infinity, $h\to h_{\max}\in\R\cup\{\infty\}$, we know $|Df|$ goes to zero and therefore from Chapter 7 of \cite{Lam} we have

	\begin{align*}
	\lm{h}{h_{\max}} \m_{\BY}(\S_h)	&= \lm{h}{h_{\max}} \frac{1}{C_n} \int_{\S_h} \( \frac{|Df|^2}{1+|Df|^2}  \) \FH_{\S_h} \, dV_{h} \\
	&=\m_{ADM}-\lm{h}{h_{\max}}\int_{E_h}\frac{R(f)}{\sqrt{1+|Df|^2}}\, dV_{\delta},
	\end{align*}
where $E_h$ denotes the region of the manifold exterior to $\S_h$ and $\m_{ADM}$ denotes the ADM mass of $(M,g)$. In particular, since $R(g)$ is integrable we have
	\bee
	\lm{h}{h_{\max}} \m_{\BY}(\S_h)=\m_{\ADM}.
	\eee
Since $\m_{BY}(\S_h)\geq 0$, we are done.
\end{proof}
	Note that the above corollary does not require a sign condition on the scalar curvature; in fact, it does not require any information about the interior of the manifold whatsoever. It suggests that the positivity of mass in the graphical case is in fact due to the graph property, rather than the non-negativity of scalar curvature.

For graphical manifolds like the ones we consider here, Lam \cite{Lam} proved the Riemannian Penrose inequality holds in all dimensions. In fact, from his work it is straightforward to show that the graphical case of the quasi-local Penrose inequality also holds for the Brown--York mass in all dimensions. We will need this result in Section \ref{Sflatconvergence}.
\begin{thm}\label{thm-QLpenrosegraphs}
Let $f:\ovr{U} \setminus U_o \To \R$ be a compact graphical manifold with non-negative scalar curvature, such that almost every level set is strictly mean convex and outer-minimizing, with $\pr U_o \neq 0$. Suppose that $h_1 \leq h_2$ are two regular values of $f$. Then we have

\be \label{eq-mono-ql}
\int_{\S_{h_2}} \frac{|Df|^2}{1+|Df^2}\FH_{\S_{h_2}} \, dV_{h_2} \geq \int_{\S_{h_1}} \frac{|Df|^2}{1+|Df^2}\FH_{\S_{h_1}} \, dV_{h_1}.
\ee
In particular, we have the following quasi-local Riemannian Penrose inequality:

\be \label{eq-BYpenrosegraphs}
\m_{\BY}(\S)  > \frac12\( \frac{|\S_o|}{\omega_{n-1}} \)^{\frac{n-2}{n-1}}.
\ee 
\end{thm}
\begin{proof}
	By the proof of Theorem 12 in \cite{Lam} and equation (5.3) therein, for a regular values $h_1 \leq h_2$ we can write
	\be 
		\int_{\Omega_{h_2} \setminus \Omega_{h_1}} \frac{R(f)}{\sqrt{1+|D f|^2}} \, dV_{\delta}=\int_{\S_{h_2}} \frac{|D f|^2}{1+|D f|^2} \FH_{\S_{h_2}} dV_{h_2}-\int_{\S_{h_1}} \frac{|D f|^2}{1+|D f|^2}\FH_{\S_{h_1}} \, dV_{h_2}.
	\ee 
	Since $R(f)\geq0$, then \eqref{eq-mono-ql} follows. Using Lemma \ref{lemma-mBYbound}, we then have
	\be 
		\m_{\BY}(\S)\geq \frac{1}{2(n-1)\omega_{n-1}}\int_{\S_o} \frac{|Df|^2}{1+|Df|^2}\FH_{\S_o} \, dV_{\S_o} = \frac{1}{2(n-1)\omega_{n-1}}\int_{\S_o} \FH_{\S_o} \, dV_{\S_o},
	\ee 
using the fact that $|Df|$ blows up on the horizon. Equation \eqref{eq-BYpenrosegraphs} then follows by the classical Minkowski inequality for mean convex and outer-minimizing surfaces\footnote{While the Minkowski inequality is a classical result, we remark that its extension to outer-minimizing surfaces was noted by Huisken \cite{H09}.} in $\R^n$, which states 
	\be \label{eq-Mink}
	\frac{1}{(n-1)\omega_{n-1}} \int_{\S_o} \FH_{\S_o} \, dV_{\S_o} \geq \( \frac{|\S_o|}{\omega_{n-1}}\)^{\frac{n-2}{n-1}}.
	\ee 

Note that equality in \eqref{eq-BYpenrosegraphs} implies equality in \eqref{eq-BYLamestimate}, which would imply $\Omega$ is isometric to a domain in $\R^n$. Since there exists an interior horizon by hypothesis, we conclude that the inequality is strict.
\end{proof}

Combining Lemma \ref{lemma-mBYbound} and Theorem \ref{thm-QLpenrosegraphs}, we immediately obtain the rigidity case for the quasi-local positive mass theorem, Theorem \ref{thm-ShiTam}, for the set $\A_U$. Note that the set $\A_U$ excludes graphs of constant fuctions (hyperplanes), so in order to state the following corollary, let $\A_C$ denote the set of compact Riemannian manifold with boundary and non--negative scalar curvature that can be isometrically embedded as a constant graph in $\R \times \ovr{U}$.

\begin{coro}
Let $(\Omega,g) \in \A_U \cup \A_C$. Then $\m_{\BY}(\pr U)=0$ if and only if  $(\Omega,g) \in  \A_C$.
\end{coro}

\begin{proof}
This follows by noting that Theorem \ref{thm-QLpenrosegraphs} implies that the Brown--York mass must vanish for almost every level set of $f$ if it vanishes for $\p U$. Then Lemma \ref{lemma-mBYbound} implies $|Df|$ must vanish on almost every level set.
\end{proof}

\section{Volume estimates} \label{section-vol-est}
The key to proving the main stability result is a volume estimate in terms of the mass, which is pursued in this section, closely following \cite{HuangLee2015}. We start by defining a special height in the graphs which will play an essential role in the volume estimates needed to prove the convergence in the flat norm. In this section we set
\begin{equation}
\m_{\BY}=\m_{\BY}(\pr U).
\end{equation}

\begin{definition} \label{def-h0}
Let $f:\ovr{U} \setminus U_o \To \R$ be a compact graphical manifold with non-negative scalar curvature, such that almost every level set is strictly mean convex and outer-minimizing, and assume $\m_{\BY}>0$. Let $h_o$ be the height defined by 
	\begin{equation}
	h_o(f) \defeq  h_o \defeq \sup\{h: \V(h) \leq 2 (1+\xi)^{\frac{n-1}{n-2}} \w_{n-1}(2\m_{\BY})^{\frac{n-1}{n-2}}\},
	\end{equation}
	where $\xi\geq1$ is some fixed constant. If the above set is empty, we set $h_o \defeq \inf(f)$.  
\end{definition}

\begin{lemma}
Let $f:\ovr{U} \setminus U_o \To \R$ be a compact graphical manifold with non-negative scalar curvature, such that almost every level set is strictly mean convex and outer-minimizing, and let $h$ be a regular value of $f$. Then for any real number $\alpha>0$, we have 
\begin{equation}
\V'(h)>\alpha^{-1}\left[\int_{\S_h}\FH_{\S_h} \, dV_h-\left(1+\alpha^{-2}\right)c_n \m_{\BY}\right],
\end{equation}
where $c_n=2(n-1)\omega_{n-1}$ and $\m_{\BY}$.
\end{lemma}
\begin{proof}Let $U_{\alpha}:=\{x\in\mathbb{R}^n:|Df(x)|\geq\alpha\}$ and $L_{\alpha}:=\{x\in\mathbb{R}^n:|Df(x)|<\alpha\}$. Then by Lemma \ref{lemma-mBYbound}, we have 
	\begin{equation}\label{eq-l1}
\begin{split}
\m_{\BY}&\geq \int_{\S_h\cap U_{\alpha}} \frac{|Df|^2}{1+|Df|^2} \FH_{S_h} \, dV_h.\\
&\geq \frac{\alpha^2}{1+\alpha^2}\int_{\S_h\cap U_{\alpha}}\FH_{S_h} \, dV_h.
\end{split}
	\end{equation}
The first variation of area of $\S_h$ in $\mathbb{R}^n$ with normal vector $\partial_h=\phi n_0$, where $\phi=|Df|^{-1}$ and $n_0=|Df|^{-1}Df$ (note that this gives a normal variation along level sets of $\graph[f]$), is
	\begin{equation}\label{eq-l2}
	\begin{split}
	\V'(h)&=\int_{\S_h}\frac{\FH_{\S_h}}{|Df|}\, dV_{h}\\
	&=\int_{\S_h\cap U_{\alpha}}\frac{\FH_{\S_h}}{|Df|}\, dV_{h}+\int_{\S_h\cap L_{\alpha}}\frac{\FH_{\S_h}}{|Df|}\, dV_{h}\\
	&>\frac{1}{\alpha}\int_{S_h\cap L_{\alpha}}\FH_{S_h} \, dV_{h}\\
	&=\frac{1}{\alpha}\left(\int_{S_h}\FH_{S_h} \, dV_{h}-\int_{S_h\cap U_{\alpha}}\FH_{S_h} \, dV_{h}\right)
	\end{split}
	\end{equation}Combining \eqref{eq-l1} and \eqref{eq-l2}, we get the result. 
\end{proof}

\begin{lemma}\label{lemma-Vprime-estimate}
Let $f:\ovr{U} \setminus U_o \To \R$ be a compact graphical manifold with non-negative scalar curvature, such that almost every level set is strictly mean convex and outer-minimizing. Assume $\m_{\BY}=\m_{\BY}(\pr U)>0$ and $h$ is a regular value of $f$ such that $\V(h)>\omega_{n-1}(2\m_{\BY})^{\frac{n-1}{n-2}}$. If $\S_h$ is strictly mean convex and outward-minimizing , then 
\begin{equation}
\V'(h)>c_n\frac{2\m_{\BY}}{3\sqrt{3}}\left[\frac{1}{2\m_{\BY}}\left(\frac{\V(h)}{\omega_{n-1}}\right)^{\frac{n-2}{n-1}}-1\right]^{\frac{3}{2}}
\end{equation}where $c_n=2(n-1)\omega_{n-1}$.
\end{lemma}
\begin{proof} By the Minkowski inequality \eqref{eq-Mink} and Lemma \ref{lemma-Vprime-estimate}, we have 
	\begin{equation}\label{eq1}
	\V'(h)>c_n\alpha^{-1}\left[\frac{1}{2}\left(\frac{\V(h)}{\omega_{n-1}}\right)^{\frac{n-2}{n-1}}-\left(1+\alpha^{-2}\right)\m_{\BY}\right]^{\frac{3}{2}}.
	\end{equation}
	
	For $\V(h)>\omega_{n-1}(2\m_{\BY})^{\frac{n-1}{n-2}}$, the right hand side of above inequality (as a function of $\alpha$) has a global maximum on the interval $(0,\infty)$ at
	\begin{equation}
	\alpha=\sqrt{3}\left[\frac{1}{2\m_{\BY}}\left(\frac{\V(h)}{\omega_{n-1}}\right)^{\frac{n-2}{n-1}}-1\right]^{-\frac{1}{2}}.
	\end{equation}Substituting this in \eqref{eq1}, we get the result.
	
	\end{proof}

In what follows, we will need to employ the following ODE comparison lemma.

\begin{lemma}\label{lemma-ODEcomp}\label{lemODEcomp}[Lemma 3.9 of \cite{HuangLee2015}] 
Let $V:[a,b]\to\R$ be a non-decreasing function satisfying $V'\geq F(V)$ almost everywhere on $[a,b]$ for some non-decreasing $F\in C^1$. Suppose further that there exists $Y\in C^2$ satisfying 
\bee 
	Y'=F(Y)\qquad \text{and} \qquad Y(a)\leq V(a).
\eee 
Then $Y\leq V$ on [a,b].
\end{lemma}

\begin{lemma} \label{lemma-supf-est}
Let $f:\ovr{U} \setminus U_o \To \R$ be a compact graphical manifold with non-negative scalar curvature, such that almost every level set is strictly mean convex and outer-minimizing. If $\m_{\BY}>0$, then for $n=3$ we have
\be \label{eq-est-maxf}
0 \leq \max(f) - h_o \leq  C_3|\pr U|^{\frac{1}{4}} \m_{\BY}^{\frac{1}{2}},
\ee 
and for $n \geq 4$, we have
\be \label{eq-est-maxf-n}
0 \leq \max(f) - h_o \leq  C_n\m_{\BY}^{\frac{1}{n-2}}\left(|\log\m_{\BY}| + |\pr U|\right),
\ee 
where $C_n$ is a dimensional constant.
\end{lemma}

\begin{proof}

Assume $h_o< \max(f)$ for if $h_o=\max(f)$ there is nothing to prove.
 We would like to use Lemma \ref{lemma-Vprime-estimate} to prove this via Lemma \ref{lemODEcomp}, so we are led to consider the differential equation
\begin{align}
\begin{split}\label{eq-ODE-Y2}
Y'(h)&=C_n\m_{\BY}\( \frac{1}{2\m_{\BY}}\(\frac{Y(h)}{\omega_{n-1}} \)^{\frac{n-2}{n-1}}-1 \)^{3/2}\\\\
Y(h_o)&=\omega_{n-1}(1+\xi)^{\frac{(n-1)}{(n-2)}}(2\m_{\BY})^{\frac{n-1}{n-2}},
\end{split}
\end{align} 
where $\xi\geq1$ is defined as in Definition \ref{def-h0}. Note that this equation has a unique smooth solution $Y$ because the right hand side is smooth for given initial data set, which we can estimate as follows. Note we have for any regular value $h\geq h_o$ of $f$, we have $\V(h)\geq 2(1+\xi)^2 \w_{n-1}(2\m_{\BY})^{\frac{n-1}{n-2}}>(1+\xi)^2 \w_{n-1}(2\m_{\BY})^{\frac{n-1}{n-2}}$. In particular, we have from Lemma \ref{lemma-Vincr} that $\V$ is non-decreasing and indeed we can apply Lemma \ref{lemma-ODEcomp} to conclude $Y(h)\leq \V(h)$ for all $h\geq h_o$.

In order to proceed, we make the substitution $p(h)=\frac{1}{2\m_{\BY}}\(\frac{Y(h)}{\omega_{n-1}} \)^{\frac{n-2}{n-1}}-1$, so we have
\bee
Y'=2\m_{\BY}p'\frac{n-1}{n-2}\omega_{n-1}\left[2\m_{\BY}(p+1)\right]^{\frac{1}{n-2}}.
\eee
This then allows us to write \eqref{eq-ODE-Y2} as a separable ODE in $p$ and $Y$ resulting in
\be \label{eq-integralestimate}
h-h_o=\frac{2}{C_n}\frac{n-1}{n-2}\omega_{n-1} (2\m_{\BY})^{\frac{1}{n-2}}\int_{h_o}^{h}p^{-\frac{3}{2}}(p+1)^{\frac{1}{n-2}}\,dp
\ee 
For the sake of exposition, set $D_n=2^{\frac{n-1}{n-2}}\frac{1}{C_n}\frac{n-1}{n-2}\omega_{n-1}$. Let $n=3$, we now estimate \eqref{eq-integralestimate} by
\bee
h-h_o\leq D_3\m_{\BY}\int_{h_o}^{h}p^{-\frac{3}{2}}(p+1)\,dp,
\eee
By integration we have

\bee 
h-h_o\leq D_3\m_{\BY}\left[\sqrt{p}-1/\sqrt{p} \right]_{h_o}^{h}.
\eee 
Noting that $p(h_o)=\xi\geq1$ we have

\bee 
h-h_o\leq D_3\m_{\BY}\sqrt{p(h)}-1/\sqrt{p(h)}\leq \widetilde C_3\sqrt{\m_{\BY}}\left[Y(h)\right]^{\frac{1}{4}}.
\eee 

Then estimate \eqref{eq-est-maxf}  follows from this and $Y(h)\leq \V(h)$ for all $h\geq h_o$. Now let $n\geq 4$, using \eqref{eq-integralestimate}, we obtain
\be \label{eq-est-4}
h-h_o\leq D_n\m_{\BY}^{\frac{1}{n-2}}\int_{h_o}^{h}p^{-3/2}(p+1)^{1/2}\,dp= D_n\m_{\BY}^{\frac{1}{n-2}}\left(k(h)-k(h_o)\right)
\ee
where
\begin{equation*}
k(h):=-2\frac{\sqrt{p(h)+1}}{\sqrt{p(h)}}+\log\left(\frac{1}{2}+p(h)+\sqrt{p(h)^2+p(h)}\right).
\end{equation*}
Since $p(h_o)=\xi\geq 1$, we have 
\begin{equation}\label{k1}
-2\frac{\sqrt{p(h)+1}}{\sqrt{p(h)}}+2\frac{\sqrt{p(h_o)+1}}{\sqrt{p(h_o)}}\leq 2\frac{\sqrt{\xi+1}}{\sqrt{\xi}}\leq 2\sqrt{2}.
\end{equation}Moreover, for all $h\geq h_o$ we have the following estimate
\begin{equation}\label{k2}
\begin{split}
\log\left(\frac{3}{2}+\sqrt{2}\right)
&\leq\log\left(\frac{1}{2}+p(h_o)+\sqrt{p(h_o)^2+p(h_o)}\right)\\
&\leq\log\left(\frac{1}{2}+p(h)+\sqrt{p(h)^2+p(h)}\right)\\
&\leq \log(2)+\log\left(p(h)+1\right).
\end{split}
\end{equation}Combining equations \eqref{k1} and \eqref{k2}, we find the following estimate
\begin{equation}\label{eq22}
\begin{split}
k(h)-k(h_o)
&\leq  2\sqrt{2}+\log(2)-\log\left(\frac{3}{2}+\sqrt{2}\right)+\log\left(p(h)+1\right)\\
& \leq 6\log\left(1+p(h)\right).
\end{split}
\end{equation}
Together with \eqref{eq-est-4}, we have 
\begin{equation*}
\begin{split}
h-h_o&\leq \widetilde C_n\m_{\BY}^{\frac{1}{n-2}}\log\left(1+p(h)\right)\\
&=\widetilde C_n\m_{\BY}^{\frac{1}{n-2}}\log\left[\m_{\BY}^{-\frac{n-1}{n-2}}\omega_{n-1}^{-1}Y(h)\right]\\
\end{split}
\end{equation*} The estimate \eqref{eq-est-maxf-n} follows from above inequality and $Y(h)\leq \V(h)$ for all $h\geq h_o$.

\end{proof}

In order to establish the convergence of volumes in Theorem \ref{thm-main} we need estimates for the volume of a compact graph controlled by its Brown--York mass. The following lemma gives the desired estimate and it follows in a very similar way to \cite{CKP}.

\begin{lemma}  \label{lemma-vol-est}
Let $f:\ovr{U} \setminus U_o \To \R$ be a compact graphical manifold with non-negative scalar curvature, such that almost every level set is strictly mean convex and outer-minimizing. The following volume estimates hold. For $n=3$,
\bee
\vol(U) \leq \vol(\graph[f]) \leq \vol(U)+ \widetilde{C}_3(2 \m_{\BY})^3 + \widetilde{C}_3(h_o - \min(f))(2 \m_{\BY})^2 + \widetilde{C}_3|\pr U|^{\frac{5}{4}} \m_{\BY}^{\frac{1}{2}},
\eee
and for $n \geq 4$, we have

\bee
\begin{split}
\vol(U) \leq \vol(\graph[f]) &\leq \vol(U)+ \widetilde C_n(2 \m_{\BY})^{\frac{n}{n-2}}  \\ &\qquad + \widetilde{C}_n(h_o - \min(f))(2 \m_{\BY})^{\frac{n-1}{n-2}} + \widetilde{C}_n\m_{\BY}^{\frac{1}{n-2}}\left(|\log\m_{\BY}| +|\pr U| \right),
\end{split}
\eee
where $\vol(\cdot)$ denotes the $n$-dimensional volume and $|\cdot|$ the $(n-1)$-dimensional volume, and $\widetilde{C}_i$ $i=3,\hdots,n$ is a dimensional constant.
\end{lemma}

\begin{proof}
To estimate the volume of $\graph[f]$ we divide it into two regions: the part of the graph lying over $\Omega_{h_o}$ and the part over $\ovr{U} \setminus \Omega_{h_o}$. For convenience we denote the part of the graph over a set $S$ as $\graph[f]\vert_{S}$, and define the following sets
\begin{align*}
V_f^-(h) &\defeq \vol(\graph[f]\vert_{\Omega_h}), \\
V_f^+(h) &\defeq \vol(\graph[f]\vert_{\ovr{U} \setminus \Omega_h}).
\end{align*}

Hence, we want to estimate $V_f^-(h_o)$ and $V_f^+(h_o)$. For this, it will be useful to note that for $h \in [\min(f),\max(f)]$, the coarea formula gives
\bee
\int_{\Omega_h} |Df|\,dV_{\delta} = \int_{\min(f)}^h \V(t) \, dt,
\eee
where $dV_{\delta}$ denotes the Euclidean volume form. In addition, the isoperimetric inequality states 
\bee
\vol(\Omega_h) \leq C_n^{-1}|\S_h|^{\frac{n}{n-1}},
\eee
where $C_n=n^{\frac{n}{n-1}}\vol(B_1(0))^{\frac{1}{n-1}}$.

Let $h^- \leq h_o$ be a regular value, then by the coarea formula, the isoperimetric inequality and the definition of $h_o$ we have the following. 

\begin{align*}
V_f^+(h^-) &= \int_{\Omega_{h^-}} \sqrt{1 + |Df|^2}\, dV_{\delta} \\
&\leq  \int_{\Omega_{h^-}}(1 + |Df|)\, dV_{\delta} \\
&= \vol(\Omega_{h^-}) + \int_{\min(f)}^{h^-} \V(t)\, dt \\
&\leq C_n\V(h)^{\frac{n}{n-1}}  + (1+\xi)^{\frac{n-1}{n-2}} \w_{n-1}(h^- - \min(f))(2 \m_{\BY})^{\frac{n-1}{n-2}}  \\
&\leq C_n(1+\xi)^{\frac{n}{n-2}} \w_{n-1}^{\frac{n}{n-1}}(2 \m_{\BY})^{\frac{n}{n-2}}+ (1+\xi)^{\frac{n-1}{n-2}} \w_{n-1}(h^- - \min(f))(2 \m_{\BY})^{\frac{n-1}{n-2}}  
\end{align*}
where $C_n$, defined above, is the dimensional constant arising from the isoperimetric inequality. By taking a sequence $\{ h_i^- \}$ of regular values such that $h_i^- \uparrow h_o$, we have
\be \label{eq-Vdown}
V_f^-(h_o) \leq C_n(1+\xi)^{\frac{n}{n-2}} \w_{n-1}^{\frac{n}{n-1}}(2 \m_{\BY})^{\frac{n}{n-2}}+ (1+\xi)^{\frac{n-1}{n-2}} \w_{n-1}(h_o - \min(f))(2 \m_{\BY})^{\frac{n-1}{n-2}}. 
\ee

Now consider a regular value $h^+ \geq h_o$. Using again the coarea formula, the fact that $\V$ is non-decreasing (Lemma \ref{lemma-Vincr}), we have
\begin{align*}
V_f^+(h^+) &= \int_{U \setminus \Omega_{h^+}}\sqrt{1 + |Df|^2} \, dV_{\delta} \\
&\leq \vol(U \setminus \Omega_{h^+}) + \int_{h^+}^{\max(f)} \V(t) \, dt  \\
&\leq \vol(U) + |\pr U|(\max(f) - h_o).
\end{align*}
Then, by taking a sequence of regular values $\{ h^+_i \}$ such that $h^+_i \downarrow h_o$, we have
\be \label{eq-Vup}
V_f^+(h_o) \leq  \vol(U) + C|\pr U|(\max(f) - h_o).
\ee

Combining \eqref{eq-Vdown} and \eqref{eq-Vup} together with Lemma \ref{lemma-supf-est} the result follows.
\end{proof}

\section{Proof of the main result}\label{Sflatconvergence}
In this section we obtain the main results of this work. We start by briefly reminding the reader the definition of the Federer--Fleming flat distance \cite{FF60} which was extended to complete metric spaces by Ambrosio--Kirchheim \cite{A-K}, and then we proceed to apply the results of Section \ref{section-vol-est}. 

Recall that a submanifold $N$ of $\R^n$ can be seen as a integral current of multiplicity one. In this case, its boundary $\pr N$ is also an integral current. Moreover, the mass of a submanifold of $\R^n$ (as integral current) is denoted by $\bM(\cdot)$, and is given by its volume (as a manifold). For simplicity, we state the definition of the flat distance in the type of metric spaces we are interested in, that is, $Z=(a,b) \times \ovr{U} $, where $-\infty  \leq a < b \leq \infty$ and $U$ is a bounded open set in $\R^n$. 

\begin{definition}
Let $T_1$ and $T_2$ be integral $k$-currents in $Z=(a,b) \times \ovr{U} $, and let $\bM_z$ denote denote the mass of a current in $Z$. The flat distance between $T_1$ and $T_2$ in $Z$ is defined as
\bee
d_F^Z(T_1,T_2)= \inf \{ \bM_Z(A) + \bM_ Z(B) \, : \, T_1 - T_2 = A + \pr B \},
\eee
where the infimum is taken over all integral $k$-currents $A$ and all integral $(k+1)$-currents $B$ in $Z= (a,b) \times \ovr{U} $, such that $T_1 - T_2 = A + \pr B$.
\end{definition}

To compute the flat distance between a compact graphical manifold over $U$ and $\{ h_o \} \times \ovr{U}$ in terms of $\m_{\BY}(\pr U)$, we follow the procedure in \cite{HuangLee2015}.

\begin{lemma} \label{lemma-flatd}
Let $f:\ovr{U} \setminus U_o \To \R$ be a compact graphical manifold with non-negative scalar curvature, such that almost every level set is strictly mean convex and outer-minimizing. Then we have for $n=3$
\bee
d_F^{\R \times \ovr{U}}(\graph[f],\{ h_o \} \times \ovr{U}) \leq \widetilde{C}_3\( (h_o - \inf(f))  \m_{\BY}^3 +  \m_{\BY}^3 +  |U|  |\pr U|^{\frac{1}{4}}\m_{\BY}^{\frac{1}{2}}\),
\eee
and for $n \geq 4$,
\bee
d_F^{\R \times \ovr{U}}(\graph[f],\{ h_o \} \times \ovr{U}) \leq \widetilde{C}_n\( (h_o - \inf(f))  \m_{\BY}^{\frac{n}{n-2}} +  \m_{\BY}^{\frac{n}{n-2}} +  |U|\m_{\BY}^{\frac{1}{n-2}}\left(|\log\m_{\BY}| + |\pr U|\right)  \),
\eee
where $\widetilde{C}_n$ is a constant that depends only in the dimension.
\end{lemma}

\begin{proof}
We will estimate the flat distance between $\graph[f]$ and $\{h_o\}\times \ovr{U}$ inside $\R \times \ovr{U}$. To do this, first let $\ovr{f}:\ovr{U} \To \R$ be defined by ``filling in'' $\ovr{U}_o$, that is, $\ovr{f}=f$ on $\ovr{U} \setminus U_o$ and $\ovr{f}=f\vert_{\pr U_o}$ on $U_o$.
 
We define the following regions (see Figure \ref{fig-flatd}, below):
\begin{align*}
B^+ &\defeq \{ (s,x) \in \R \times U \, | \, h_o \leq s \leq \ovr{f}(x)  \}, \\
B^- &\defeq \{ (s,x) \in \R \times U \, | \, \ovr{f}(x) \leq s \leq h_o  \}, \\
A &\defeq \{ (s,x) \in \R \times U \, | \, s=f(y)\textnormal{ for $y \in \pr U_o$}, x \in \ovr{U_o}  \}.
\end{align*}

Note that $B^+$ is the region above $\{ h_o \} \times \ovr{U}$ and below $\graph[\ovr{f}]$, $B^-$ is the region below $\{ h_o \} \times \ovr{U}$ and above $\graph[\ovr{f}]$, and $A$ is the ``filled-in'' boundary inside the hyperplane that contains them. Here we are assuming that $h_o \in (\min(f),\max(f))$ and $U_o \neq \emptyset$, but if this is not the case, the only difference is that some of the above sets might be empty. 

Taking $f$ with the upward orientation and $A$ with the downward orientation, we have that, as integral currents, $\graph[f] - A=\graph[\ovr{f}]$. Taking $\{h_o\}\times \ovr{U}$ and $B^+$ with the upward orientation, and $B^-$ with the downward orientation, we have $\graph[\ovr{f}] - \{h_o\}\times \ovr{U}=\pr B + A$, where $B=B^++B^-$. Hence to compute the flat distance between $\graph[f]$ and $\{h_o\}\times \ovr{U}$ inside $\R \times \ovr{U}$, we need to compute $\bM(B)=\bM(B^+) + \bM(B^-)$ (since $B^+$ and $B^-$ are disjoint) and $\bM(A)$. In what follows we will denote all constants by $C$, but we remark that although they may be different they only depend on the dimension.

For any regular value $h \leq h_o$, by definition of $h_o$ we have

\bee
 \V(h)\leq (1+\xi)^{\frac{n-1}{n-2}} \w_{n-1}(2 \m_{\BY})^{\frac{n-1}{n-2}} 
\eee
Then, by the isoperimetric inequality, we have $|\Omega_h| \leq C (1+\xi)^{\frac{n}{n-2}}(\w_{n-1})^{\frac{n}{n-1}} \m_{\BY}^{\frac{n}{n-2}}$. Then we can estimate $\bM(B^-)$ as follows.
\begin{align*}
\bM(B^-) &\leq \int_{\inf(f)}^{h_o} C (1+\xi)^{\frac{n}{n-2}}(\w_{n-1})^{\frac{n}{n-1}} \m_{\BY}^{\frac{n}{n-2}} \, dh \\
&\leq C (h_o - \inf(f)) \m_{\BY}^{\frac{n}{n-2}}.
\end{align*}

By the quasi-local Riemannian Penrose inequality for graphs (Theorem \ref{thm-QLpenrosegraphs}), we know that
\be 
\m_{\BY} > \frac12\( \frac{|\S_o|}{\omega_{n-1}} \)^{\frac{n-2}{n-1}}.
\ee 
Using again the isoperimetric inequality we have
\bee
\bM(A) \leq  C \m_{\BY}^{\frac{n}{n-2}}.
\eee

Recall that by Lemma \ref{lemma-supf-est}, we have $0 \leq \sup(f) - h_o \leq C |\pr U|^{\frac{1}{4}}\m_{\BY}^{\frac{1}{2}}$ for $n=3$. It follows readily that
\begin{align*}
\bM(B^+) &\leq |U| ( \sup(f) - h_o ) \\
&\leq C |U|  |\pr U|^{\frac{1}{4}}\m_{\BY}^{\frac{1}{2}}.
\end{align*}

Therefore, we obtain
\bee
d_F^{\R \times \ovr{U}}(\graph[f],\{ h_o \} \times \ovr{U}) \leq C\( (h_o - \inf(f))  \m_{\BY}^3 +  \m_{\BY}^3 +  |U|  |\pr U|^{\frac{1}{4}}\m_{\BY}^{\frac{1}{2}}\).
\eee

Finally, the case $n \geq 4$ follows in the same way using that
\bee
\bM(B^+) \leq C |U| \m_{\BY}^{\frac{1}{n-2}}\left(|\log\m_{\BY}| + |\pr U|\right).
\eee
\end{proof}

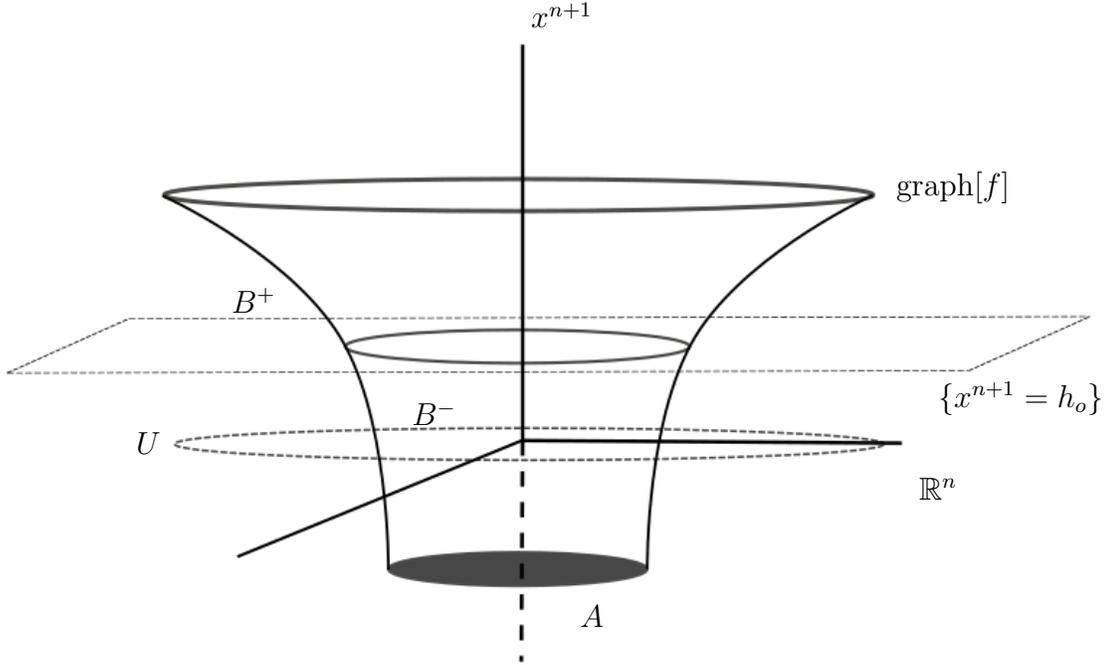
\begin{figure}[h!]

	\begin{tikzpicture}
    \node[anchor=south west,inner sep=0] at (0,0) {\includegraphics[scale=.9]{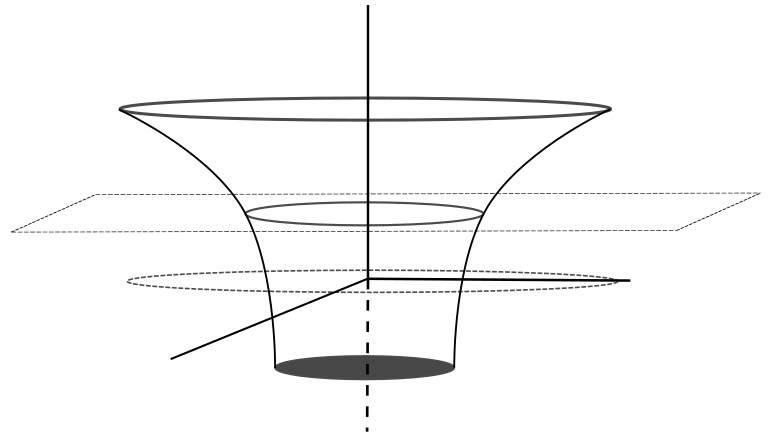}};
    \node at (13.7,3.7) {$\{x^{n+1}=h_o\}$};
    \node at (12.6,2.5) {$\R^n$};
    \node at (7.6,8.8) {$x^{n+1}$};
    \node at (12.8,6.5) {$\graph[f]$};
    \node at (8,0.8) {$A$};
    \node at (5.9,3.5) {$B^-$};
    \node at (3.5,5) {$B^+$};
    \node at (2.1,3.1) {$U$};
\end{tikzpicture}
	\caption{Depiction of the regions $B^+$, $B^-$, and $A$, described above.}
\label{fig-flatd}
\end{figure}

We now state Theorem \ref{thm-main} in a precise way. 

\begin{thm} \label{thm-main2}
Let $\{(\Omega_j,g_j)\}$ be a sequence of admissible compact graphs over $U$, that is $(\Omega_j,g_j) \in \A_U$, and let $f_j:\ovr{U}\setminus U_o^j \to \R$ be the function realizing $(\Omega_j,g_j)$ as a graph in $\R^{n+1}$ ($n \geq 3$). Normalize $f_j$ so that $h_o(f_j)=0$. Let $\m_{\BY}^{f_j}=\m_{\BY}^{f_j}(\pr U)$. We consider the following two cases:
\begin{enumerate}
\item suppose that $\inf(f_j)(\m_{\BY}^{f_j})^{\frac{n}{n-2}} \to 0$ as $\m_{\BY}^{f_j} \to 0$, or
\item that $(\Omega_j,g_j) \in \A_U^D$.
\end{enumerate}
If $\m_{\BY}(\Omega_j) \to 0$, then $\{(\Omega_j,g_j)\}$ converges to the region $\ovr{U}$ in $\R^3$ with respect to the Federer--Fleming flat distance. Moreover $\vol(\graph[f_j]) \to \vol(U)$ as $\m_{\BY}^{f_j} \to 0$.
\end{thm}

\begin{remark}
The normalization so that $h_o(f_j)=0$ for all $j$ is achieved by applying vertical translations.
\end{remark}

\begin{proof}
Let $\{ f_i \}$ be the sequence given in the statement. Normalize this sequence by a vertical translation so that $h_o(f_i)=0$. By Lemma \ref{lemma-flatd}, we have for $n=3$ 

\bee
d_F^{\R \times \ovr{U}}(\graph[f_j],\{ 0\} \times \ovr{U}) \leq \widetilde{C}_3\( \inf(f_j) \m_{\BY}^3 +  \m_{\BY}^3 +  |U|  |\pr U|^{\frac{1}{4}}\m_{\BY}^{\frac{1}{2}}\),
\eee
and for $n \geq 4$,
\bee
d_F^{\R \times \ovr{U}}(\graph[f_j],\{ 0 \} \times \ovr{U}) \leq \widetilde{C}_n\( \inf(f_j)  \m_{\BY}^{\frac{n}{n-2}} +  \m_{\BY}^{\frac{n}{n-2}} +  |U|\m_{\BY}^{\frac{1}{n-2}}\left(|\log\m_{\BY}| + |\pr U|\right)  \).
\eee

Clearly, if $\inf(f_j)(\m_{\BY}^{f_j})^{\frac{n}{n-2}} \to 0$ as $\m_{\BY}^{f_j} \to 0$ or $(\Omega_j,g_j) \in \A_U^D$, the flat distance between $\graph[f_j]$ and $\{ 0\} \times \ovr{U}$ converges to 0 as $\m_{\BY}^{f_j} \to 0$ and the convergence with respect to the flat distance follows. The convergence of the volumes follows in a similar way by applying Lemma \ref{lemma-vol-est}.
\end{proof}


\begin{thebibliography}{99}
	
\bibitem{Aghil} Alaee, A., Lesourd, M., and Yau, S.-T., \textit{A localized spacetime Penrose inequality and horizon detection with quasi-local mass}, preprint, arXiv:1912.01581.

\bibitem{Allen1} Allen, B., {\sl IMCF and the Stability of the PMT and RPI
	Under $L^2$ Convergence}, Ann. Henri Poincar\'e, {\bf 19} (2018), 1283--1306.

\bibitem{Allen2} Allen, B., {\sl Sobolev stability of the PMT and RPI using IMCF}, preprint, arXiv:1808.07841.	

\bibitem{Allen3} Allen, B., {\sl Stability of the {PMT} and {RPI} for asymptotically hyperbolic manifolds foliated by {IMCF}}, J. Math. Phys. {\bf 59}(8) (2018), 082501, 18 pp.

\bibitem{AB19} Allen, B. and Burtscher, A., \textit{Properties of the Null Distance and Spacetime Convergence}, preprint, arXiv:1909.04483.	

	
\bibitem{A-K} Ambrosio, L. and Kirchheim, B. \textit{Currents in metric spaces}, Acta Math., {\bf 185}(1) (2000), 1--80.
	
	
\bibitem{ADM} Arnowitt, R., Deser, S. and Misner, C. W., {\sl Coordinate invariance and energy expressions in general relativity}, Phys. Rev., {\bf 122}(3) (1961), 997--1006.

\bibitem{Bartnik-86} Bartnik, R., \textit{The mass of an asymptotically flat manifold}, Comm. Pure Appl. Math.  \textbf{39}(5) (1986), 661--693.

\bibitem{Bray01} Bray, H. L., \textit{ Proof of the Riemannian Penrose inequality using the positive mass theorem}, 
J. Differential  Geom., {\bf 59}(2) (2001), 177--267.

\bibitem{B-F} Bray, H. L. and Finster, F., \textit{Curvature estimates and the positive mass theorem}, Comm. Anal. Geom. {\bf 10}(2) (2002), 291--306.

\bibitem{Bryden} Bryden, E., \textit{ Stability of the Positive Mass Theorem for Axisymmetric Manifolds}, preprint, arXiv:1806.02447.

\bibitem{BKS19} Bryden, E., Khuri, M. and Sormani, C., \textit{Stability of the Spacetime Positive Mass Theorem in Spherical Symmetry}, preprint, arXiv:1906.11352.

\bibitem{CabreraPacheco19} {Cabrera Pacheco}, A. J. \textit{ On the stability of the positive mass theorem for asymptotically hyperbolic graphs}, Ann. Glob. Anal. Geom. \textbf{56}(3) (2019), 443--463.

\bibitem{CKP} {Cabrera Pacheco}, A. J., Ketterer, C., and Perales, P., \textit{ Stability of graphical tori with almost nonnegative scalar curvature}, preprint, arXiv:1902.03458.

\bibitem{Cohn-Vossen} Cohn-Vossen, S., \textit{Unstarre geschlossene Fl\"ache}, Math. Ann. \textbf{102} (1929), 159--184.

\bibitem{Corvino05} Corvino, J., \textit{A note on asymptotically flat metrics on $\R^3$ which are scalar-flat and admit minimal spheres}, Proc. Amer. Math. Soc. \textbf{133}(9) (2014), 3669--3678.

\bibitem{D-G-S-stability} Dahl, M. and Gicquaud, R. and Sakovich, A., \textit{Asymptotically hyperbolic manifolds with small mass}, Comm. Math. Phys. {\bf 325}(2) (2014), 757--801.

\bibitem{DR-rigidity} Dajczer, M. and Rodriguez, L., \textit{Infinitesimal rigidity of Euclidean submanifolds}, Ann. I. Fourier \textbf{40}(4) (1990), 939--949.	

\bibitem{EMW} Eichmair, M., Miao, P., and Wang, X., \textit{Extension of a theorem of Shi and Tam}, Calc. Var. \textbf{43}(1-2) (2012), 45--56.	

	
\bibitem{FF60} Federer, H. and Fleming W. H., \textit{ Normal and integral current}, Ann. Math. \textbf{72}(2) (1960), 458--520.

\bibitem{Finster} Finster, F., \textit{A level set analysis of the {W}itten spinor with applications to curvature estimates}, Math. Res. Lett. {\bf 16}(1), (2009), 41--55.

\bibitem{FK} Finster, F. and Kath, I., \textit{Curvature estimates in asymptotically flat manifolds of positive scalar curvature}, Comm. Anal. Geom. {\bf 10}(5) (2002) 1017--1031.

\bibitem{HuangLee2015}	L.-H. Huang and D. A. Lee, \textit{Stability of the positive mass theorem for graphical hypersurfaces of Euclidean
	space}, Comm. Math. Phys. \textbf{337}(1) (2015), 151--169.

\bibitem{HLS} L.-H. Huang, D. A. Lee and  and C. Sormani \textit{Intrinsic flat stability of the positive mass theorem for graphical hypersurfaces of {E}uclidean space}, J. Reine Angew. Math. {\bf 727} (2017), 269--299.

	
	\bibitem{HuangWu2013} L.-H. Huang and D. Wu,. \textit{Hypersurfaces with nonnegative scalar curvature}. J.
	Differential Geom., \textbf{95}(2) (2013), 249--278.
	
\bibitem{HuangWu2015} L.-H. Huang and D. Wu, \textit{The equality case of the Penrose inequality for asymptotically flat graphs}, Trans. Amer. Math. Soc. \textbf{367}(1) (2015), 31–47.
	
	\bibitem{H09} Huisken, G., Inverse mean curvature flow and isoperimetric inequalities, video available at \url{https://video.ias.edu/node/233} (2009)
	
\bibitem{H-I01} Huisken, G. and Ilmanen, T., \textit{The inverse mean curvature flow and the {R}iemannian {P}enrose inequality},
J. Differential  Geom., \textbf{59}(3) (2001), 353--437.


\bibitem{Lam} Lam, M.-K. G. \textit{The Graph Cases of the Riemannian Positive Mass and
Penrose Inequalities in All Dimensions}, ProQuest LLC, Ann Arbor, MI, 2011, Thesis
(Ph.D.)–Duke University. MR 2873434

\bibitem{Lee} Lee, D. A., \textit{On the near-equality case of the positive mass theorem}, Duke Math. J. {\bf 148}(1) (2009), 63--80.


\bibitem{LeeSormani} Lee, D. and Sormani, C., \textit{ Stability of the positive mass theorem for rotationally symmetric Riemannian manifolds}, J. Reine Angew. Math., 686 (2014), 187--220.

\bibitem{LM} Lu, S. and Miao, P., \textit{ Minimal hypersurfaces and boundary behavior of compact manifolds with nonnegative scalar curvature}, J. Differential Geom., \textbf{113}(3) (2019), 519--566.

\bibitem{MeeksSimonYau} Meeks, W. III, Simon, L., and  Yau, S.-T., \textit{Embedded minimal surfaces, exotic spheres, and manifolds with positive ricci curvature}, Ann. of Math. (2) \textbf{116}(3) (1982), 621–659.

\bibitem{MM19} McCormick, S. and Miao, P.,  
\textit{ On a Penrose-like inequality in dimensions less than eight}, Int. Math. Res. Not., \textbf{2019}(7) (2019).

\bibitem{Miao} Miao, P., \textit{Positive mass theorem on manifolds admitting corners along a hypersurface}, Adv. Theor. Math. Phys., \textbf{6}(6) (2002), 1163-1182.

		\bibitem{Miao09}
Miao, P., \textit{ On a Localized Riemannian Penrose Inequality}, Comm. Math. Phys., 
\textbf{292}(1) (2009), 271--284.

		\bibitem{Nirenberg}
Nirenberg, L., \textit{The Weyl and Minkowski problems in differential geometry in the large}, Commun. Pure Appl. Math, 
\textbf{6} (1953), 337--394.

		\bibitem{Pogorelov}
Pogorelov, A. V., \textit{Regularity of a convex surface with given Gaussian curvature}, Mat. Sb. (N.S.), \textbf{31(73)}
(1952), 88--103.

\bibitem{S-S} Sakovich, A. and Sormani, C., \textit{Almost rigidity of the positive mass theorem for asymptotically hyperbolic manifolds with spherical symmetry}, Gen. Relativity Gravitation {\bf 49}(9) (2017).

	\bibitem{Schoen-Yau79} Schoen, R. and  Yau, S.-T.,
\textit{ On the proof of the positive mass conjecture in general relativity},
Commun.  Math. Phys. \textbf{65}(1) (1979), 45--76.

		\bibitem{SY-17} Schoen, R. and Yau, S.-T.,
\textit{ Positive scalar curvature and minimal hypersurface singularities},
preprint, arXiv:1704.05490.

		\bibitem{ShiTam02} Shi, Y. and  Tam, L.-F.,
		\textit{ Positive mass theorem and the boundary behaviors of compact manifolds with nonnegative scalar curvature},
		J. Differential. Geom. \textbf{62}(1) (2002), 79--125.

		\bibitem{ShiTam07} Shi, Y. and  Tam, L.-F.,
		\textit{ Quasi-local mass and the existence of horizons},
		Comm. Math. Phys. \textbf{274}(2) (2007), 277--295.

		\bibitem{SW11} Sormani, C. and Wenger, S.
\textit{ The intrinsic flat distance between Riemannian manifolds and other integral current spaces},
J. Differential. Geom. \textbf{87}(1) (2011), 117--199.

		\bibitem{QLMLRR} Szabados, L.
\textit{ Quasi-Local Energy-Momentum and Angular Momentum in General Relativity},
Living Rev. Relativ. \textbf{12}(4) (2009).

	\bibitem{Witten81} Witten, E.,
\textit{ A new proof of the positive energy theorem},
Commun.  Math. Phys. \textbf{80}(3) (1981), 381--402.


\end{thebibliography}
\end{document}